\documentclass[11pt,letterpaper]{amsart}
\usepackage{amsmath}
\usepackage{amstext}
\usepackage{amssymb}
\usepackage{amsfonts}
\usepackage{enumerate}
\usepackage{amsthm}
\usepackage{mathrsfs}
\usepackage{bbm}

\numberwithin{equation}{section}

\newtheorem{theorem}{Theorem}[section]

\newtheorem{proposition}[theorem]{Proposition}
\newtheorem{corollary}[theorem]{Corollary}

\theoremstyle{definition}
\newtheorem{remark}[theorem]{Remark}
\newtheorem{definition}[theorem]{Definition}

\DeclareMathOperator{\rank}{rank}

\DeclareMathOperator{\codim}{codim}
\DeclareMathOperator{\ada}{ada}
\DeclareMathOperator{\diag}{diag}
\DeclareMathOperator{\minimize}{minimize}

\newcommand{\n}[1]{ \left\|#1\right\| }

\newcommand{\R}{{\mathbb{R}}}

\newcommand{\pair}[2]{{\langle #1, #2 \rangle}}

\begin{document}

\title[Stability of low-rank matrix recovery]{\textbf{Stability of low-rank matrix recovery and its connections to Banach space geometry}}
\author{Javier Alejandro Ch\'avez-Dom\'{\i}nguez}
\address{Department of Mathematics,
University of Texas at Austin,
2515 Speedway Stop C1200,
Austin, TX 78712-1202.}
\curraddr{
Instituto de Ciencias Matem\'aticas,
CSIC-UAM-UC3M-UCM, C/Nicol\'as Cabrera, n$^\circ$~13-15,
Campus de Cantoblanco, UAM,
28049, Madrid, Spain.
}
\email{jachavezd@math.utexas.edu}

\author{Denka Kutzarova}
\address{Institute of Mathematics, Bulgarian Academy of Sciences, Sofia, Bulgaria.}
\curraddr{Department of Mathematics,
University of Illinois at Urbana-Champaign,
1409 W. Green Street Urbana, IL 61801.}
\email{denka@math.uiuc.edu}

\begin{abstract}
	There are well-known relationships between compressed sensing and the geometry of the finite-dimensional $\ell_p$ spaces.
	A result of Kashin and Temlyakov \cite{Kashin-Temlyakov} can be described as a characterization of the stability of the recovery of sparse vectors via $\ell_1$-minimization in terms of the Gelfand widths of certain identity mappings between finite-dimensional $\ell_1$ and $\ell_2$ spaces,
	whereas a more recent result of Foucart, Pajor, Rauhut and Ullrich \cite{Foucart-Pajor-Rauhut-Ullrich} proves an analogous relationship even for $\ell_p$ spaces with $p < 1$. 
	In this paper we prove what we call matrix or noncommutative versions of these results:
	we characterize the stability of low-rank matrix recovery via Schatten $p$-(quasi-)norm minimization
	in terms of the Gelfand widths of certain identity mappings between finite-dimensional Schatten $p$-spaces.
\end{abstract}

\maketitle

\section{Introduction}

A mathematical problem that appears often in real-world situations is the following: we wish to recover a high-dimensional vector $x \in \R^N$ from a measurement $Ax$ where $A : \R^N \to \R^m$ is a linear map and $m$ is smaller than $N$. As stated the problem of course cannot be solved, but that changes if we have the additional condition that the unknown vector $x$ is \emph{sparse}, i.e. it has a small number of non-zero coordinates.
This is the subject matter of compressed sensing, a very active area of research with numerous applications; the book \cite{Foucart-Rauhut} is a recent comprehensive reference.
Formally, this \emph{sparse recovery problem} can be stated as
\begin{equation}\label{sparse-recovery-problem}
	\min \n{x}_0 \quad \text{subject to } Ax = y,
\end{equation}
where $\n{\cdot}_0$ represents the number of nonzero coordinates of a vector.
This is an NP-hard \cite{Natarajan} and non-convex problem, so we are interested in conditions (especially on the map $A$) that would allow us to solve an easier problem and still arrive to the right solution.
In that spirit, a basic technique in compressed sensing is that of $\ell_1$-minimization: if the vector $x$ is sparse enough, then minimizing $\n{x'}_{\ell_1}$ over all vectors $x'$ so that $Ax' = Ax$ actually allows us to recover $x$. Formally, instead of problem \eqref{sparse-recovery-problem} we are considering its \emph{convex relaxation}
\begin{equation}\label{min-ell1-problem}
	\min \n{x}_{\ell_1} \qquad \text{subject to } Ax = y.
\end{equation}
Aditionally, one can consider the analogous problem of $\ell_p$-minimization.

In practice the unknown vectors are not necessarily sparse, but are close to sparse ones. Thus for any method of recovery it is of utmost importance to investigate its stability, that is, having a control on the distance between the original vector and its reconstruction in terms of the distance from the original vector to the sparse vectors.
It turns out that the stability of sparse vector recovery through $\ell_p$-minimization has connections to the Banach-space geometry of finite-dimensional $\ell_p$-spaces.
More generally, it is known that there are connections between recovery -- in particular the compressed sensing model -- and geometric quantities called \emph{Gelfand widths}, see e.g. \cite{Novak,Donoho,Cohen-Dahmen-DeVore,Kashin-Temlyakov}.

In many practical situations, there is extra structure in the space of unknown vectors.
A good example is the famous \emph{matrix completion problem} (also known as the \emph{Netflix problem}), where the unknown is a matrix and the measurement map gives us a subset of its entries.
In this case sparsity gets replaced by the more natural condition of having low rank, and the last few years have witnessed an explosion of work in this area.
In what follows, $M_N$ will denote the space of $N \times N$ real-valued matrices.
We now consider a linear operator $\mathcal{A} : M_N \to \R^m$, and a fixed vector $y \in \R^m$.
The low-rank recovery problem can thus be stated as the problem of finding the solution to
\begin{equation}\label{low-rank-problem}
	\min \rank(X) \qquad \text{subject to } \mathcal{A} X = y.
\end{equation}
This is again an NP-hard problem, so once again we would like to replace it by another one which is simpler to solve but has the same solution.

In noncommutative functional analysis the Schatten $p$-spaces are usually considered to be the counterparts of the classical $\ell_p$ spaces (recall that the Schatten $p$-norm of a matrix $X$ is the $\ell_p$-norm of its vector of singular values), so from that point of view it is natural to wonder whether the Schatten $p$-norm minimization approach can work in the matrix context.
We would like to consider operators $\mathcal{A}$ for which the previous problem is equivalent to
\begin{equation}\label{min-nuclear-problem}
	\min \n{X}_{S_1} \qquad \text{subject to } \mathcal{A}X = y.
\end{equation}
Where $\n{X}_{S_p}$ denotes the Schatten $p$-norm of the matrix $X \in M_N$.
This has already been studied in several situations of interest, with the idea going back to the Ph.D. thesis of M. Fazel \cite{Fazel-thesis}.
Schatten $1$-norm (also known as nuclear norm) minimization in the particular case  of the matrix completion problem was studied by Cand\`es and Recht \cite{Candes-Recht} (and later on Cand\`es and Tao \cite{Candes-Tao} gave optimality results quantifying the minimum number of entries needed to recover a matrix of low rank exactly by any method whatsoever, and showed that nuclear norm minimization is nearly optimal).
Plenty of concepts from the classical theory of compressed sensing have found matrix counterparts:
Cand\`es and Recht \cite{Candes-Recht} use the idea of \emph{coherence};
Recht, Fazel and Parrilo \cite{Recht-Fazel-Parrilo} used the matrix version of the \emph{restricted isometry property} \cite{Recht-Fazel-Parrilo};
whereas both Recht, Xu and Hassibi \cite{Recht-Xu-Hassibi} and Fornasier, Rauhut and Ward \cite{Fornasier-Rauhut-Ward}
consider \emph{null-space conditions};
the \emph{spherical section property} was used by Dvijotham and Fazel \cite{Dvijotham-Fazel} and
Oymak, Mohan, Fazel and Hassibi \cite{Oymak-Mohan-Fazel-Hassibi}.

One thing that does not appear to have been explicitly studied in the matrix context is the aforementioned relationship to Gelfand widths.
Recall that the Gelfand $k$-width of a subset $K$ of a normed space $E$ is defined as
$$
d^k(K,E) := \inf \left\{  \sup_{x \in K \cap L} \n{x}_E \; : \; L \text{ subspace of } E \text{ with } \codim(L) \le k \right\}.
$$
A closely related concept that is more commonly used in Banach space theory is that of a \emph{Gelfand number}:
if $T : X \to Y$ is a linear operator between normed spaces, its $k$-Gelfand number is defined by
\begin{align*}
	c_k(T) &:= \inf\left\{ \sup_{x \in L, \n{x}\le 1 } \n{Tx} \; : \; L \text{ subspace of } X \text{ with } \codim(L) < k \right\}.
\end{align*}
The speed of convergence to zero of the sequence of Gelfand numbers $(c_k(T))_{k=1}^\infty$ is a measure of the compactness of the operator $T$, and is an example of a sequence of $s$-numbers; see \cite{Pietsch-s-numbers,Konig-eigenvalues} for more details.
In the cases under consideration in this paper the concepts of Gelfand numbers and Gelfand widths actually coincide (up to a small shift in the index), so we will freely use them both depending on the particular context.
It should be mentioned that there is a general concept of Gelfand width for a linear map that is not always the same as the corresponding Gelfand number (see \cite[Sec. 6.2.6]{Pietsch-history} for the details),
but both concepts do coincide in nice situations (see \cite{Edmunds-Lang}).

The work of Kashin and Temlyakov \cite{Kashin-Temlyakov} made more precise the already-known connection between compressed sensing and the Kashin-Garnaev-Gluskin \cite{Kashin,Garnaev-Gluskin} result that calculates the $m$-Gelfand numbers of the identity map from $\ell_1^N$ to $\ell_2^N$, namely
$$
c_{m+1}(id : \ell_1^N \to \ell_2^N) \le C \sqrt{ \frac{1+\log(N/m)}{m} }.
$$
In a nutshell, the main result of Kashin and Temlyakov shows that the stability of sparse recovery via $\ell_1$-minimization is equivalent to the kernel of the measurement map being a ``good'' subspace where the Gelfand number of a certain order is achieved.
This idea was taken further by Foucart, Pajor, Rauhut and Ullrich \cite{Foucart-Pajor-Rauhut-Ullrich}, who used compressed sensing ideas to calculate the  Gelfand numbers of identity maps from $\ell_p^N$ to $\ell_q^N$ for $0 < p \le 1$, $p < q \le 2$.

In this paper, we prove matrix versions of the aforementioned theorems, relating the stability of low-rank matrix recovery to the Gelfand numbers of identity maps between finite-dimensional Schatten $p$-spaces.
As far as we know the only part of our results that is already written down in the literature is the following analogue of the Kashin-Garnaev-Gluskin result due to Carl and Defant \cite[p. 252]{Carl-Defant-tensor-products}, namely the calculation of the $m$-Gelfand numbers of the identity map from $S_1^N$ to $S_2^N$: for $1 \le m \le N^2$,
$$
c_m(id : S_1^N \to S_2^N) \asymp \min\left\{1, \frac{N}{m} \right\}^{1/2}.
$$
Here and in the rest of the paper, the symbol $\asymp$ means that the quantities on the left and the right are equivalent up to universal constants. If we want to emphasize the dependance of the constants on some parameters, those will appear as subindices of the equivalence symbol ($\asymp_{p,q}$, for example).

The rest of this paper is organized as follows.
In section \ref{sec-preliminaries} we introduce our notation and state several known results that will be needed in the sequel.
In section \ref{sec-KT} we show the first relationships between the stability of low-rank matrix recovery and the geometry of Banach spaces, by proving a matrix version of the Kashin-Temlyakov theorem. 
Section \ref{sec-FL} contains a technical result, a matrix version of the main theorem from \cite{Foucart-Lai} that gives conditions on the measurement map $\mathcal{A}$ that guarantee the stability of the Schatten $p$-minimization scheme.
A very similar theorem was recently obtained independently by Liu, Huang and Chen \cite{Liu-Huang-Chen},
though our proof is different and we require a weaker hypothesis.
In the final section, the technical result from Section \ref{sec-FL} is used to calculate the Gelfand numbers of the identity maps from $S_p^N$ to $S_q^N$ for $0 < p \le 1$, $p < q \le 2$ in the spirit of the work of Foucart, Pajor, Rauhut and Ullrich.

\section{Notation and preliminaries}\label{sec-preliminaries}

In this paper we will only consider square matrices, but all the results can be adapted to rectangular ones.
For $p>0$ we will denote by $S_p^N$ the space of
$N \times N$ matrices with the Schatten $p$-quasi-norm, given by 
$$
\n{X}_{S_p} = \Big(\sum_{i=1}^N |\sigma_i|^p\Big)^{1/p},
$$
where $(\sigma_i)_{i=1}^N$ is the vector of singular values of the matrix $X$.
Similarly, $S_{p,\infty}^N$ will denote the space of
$N \times N$ matrices with the weak-Schatten-$p$-quasi-norm given by
$$
\n{X}_{S_{p,\infty}} = \max_{1 \le k \le N} k^{1/p}|\sigma^*_k|
$$
where $(\sigma^*_i)_{i=1}^N$ is the non-increasing rearrangement of $(\sigma_i)_{i=1}^N$.
For any quasi-normed space $X$, $B_X$ will denote its unit ball.

We will need to consider the best $s$-rank approximation error in the Schatten $p$-quasi-norm,
$$
\rho_s(X)_{S_p} := \inf \big\{ \n{ X - Y }_{S_p} \; : \; \rank(Y) \le s \big\}.
$$ 
It is well known that the infimum is actually attained at the $s$-spectral truncation $Y =X_{[s]}$
(that is, keeping only the $s$ largest singular values in the singular value decomposition).

Given a linear map $\mathcal{A} : M_N \to \R^m$ and a vector $y \in \R^m$, for $0 < p \le 1$ we will denote by $\Delta_p(y)$ a solution to
$$
\minimize \n{Z}_{S_p} \quad \text{ subject to } \mathcal{A}Z = y.
$$
That is, $\Delta_p$ is the Schatten $p$-quasi-norm minimization reconstruction map. The map $\Delta_p$ of course depends on the measuring map $\mathcal{A}$, but for simplicity we do not make this dependence explicit in the notation.

\subsection{The Restricted Isometry Property}

The Restricted Isometry Property (RIP) for a linear map $A :\R^N \to \R^m$ was introduced by Cand\`es and Tao \cite{Candes-Tao-RIP}, and
quickly became a key concept in the analysis of sparse recovery via $\ell_p$-norm minimization. The $s$-order restricted isometry constant of such a map is the smallest $\delta > 0$ such that for every vector $x \in \R^N$ of sparsity at most $s$,
$$
(1-\delta) \n{x}^2_{\ell_2} \le \n{Ax}^2_{\ell_2} \le (1+\delta) \n{x}^2_{\ell_2}.
$$
The importance of the RIP stems from the fact that small restricted isometry constants imply exact recovery via $\ell_p$-quasi-norm minimization for $0 < p \le 1$, 
and it should be noted that it is well known that random choices of the matrix $A$ give small RIP constants of order $s$, as long as $m$ is at least of the order of $s \ln(e N /s)$ \cite{Candes-Tao-NearOptimal,Baraniuk-Davenport-DeVore-Wakin,Mendelson-Pajor-Tomczak}.

The version of the RIP for matrix recovery was introduced by Recht, Fazel and Parrilo \cite{Recht-Fazel-Parrilo}, and is as follows:
a linear map $\mathcal{A} : M_N \to \R^m$ is said to have the Restricted Isometry Property of rank $s$ with constant $\delta>0$ if for every matrix $Z \in M_N$ of rank at most $s$,
$$
(1-\delta) \n{Z}^2_{S_2} \le \n{\mathcal{A}Z}^2_{\ell_2} \le (1+\delta) \n{Z}^2_{S_2}.
$$
The best such constant is denoted by $\delta_s(\mathcal{A})$.

Just as in the vector case, random constructions give small RIP constants.
The next result follows from \cite[Thm. 2.3]{Candes-Plan}, and will be very important for us in the sequel.

\begin{theorem}\label{thm-existence-RIP-matrices}
	Given a prescribed $\delta\in(0,1)$, there is a constant $C_\delta$ such that if the entries of the map $\mathcal{A}$ (seen a matrix with respect to the canonical bases in $M_N$ and $\R^m$) are independent gaussians with mean zero and variance $1/m$, then with positive (even overwhelming) probability $\delta_s(\mathcal{A}) \le \delta$ holds provided that 
	\begin{equation}\label{mRIP}
		m \ge C_\delta sN.
	\end{equation}
\end{theorem}


\section{A noncommutative Kashin-Temlyakov theorem}\label{sec-KT}

We will prove a matrix version of the Kashin-Temlyakov characterization of the stability of sparse recovery via $\ell_1$-norm minimization in terms of widths.
To this end, we define 
three properties modeled after the ones studied in \cite{Kashin-Temlyakov}.

\begin{definition}
	Let $N^2 > m$ and $\mathcal{A} : M_N \to \R^m$ a linear operator. We say that $\mathcal{A}$ has a:
	\begin{enumerate}[(a)]
	\item  \emph{Matrix Strong Compressed Sensing Property (MSCSP)}  if for any $X \in M_N$ we have
	\begin{equation*}\label{eqn-MSCSP}
	\n{X - \Delta_1(\mathcal{A}X) }_{S_2} \le Cs^{-1/2} \rho_s(X)_{S_1}
	\end{equation*}
	for $s \asymp m/N$.
	\item \emph{Matrix Weak Compressed Sensing Property (MWCSP)} if for any $X \in M_N$ we have
	\begin{equation*}\label{eqn-MWCSP}
	\n{X - \Delta_1(\mathcal{A}X)}_{S_2} \le Cs^{-1/2} \n{X}_{S_1}
	\end{equation*}
	for $s \asymp m/N$.
	\item \emph{Matrix Width Property (MWP)} if for any $X \in \ker(\mathcal{A})$,
	\begin{equation}\label{eqn-MWP}
	\n{X}_{S_2} \le C (N/m)^{-1/2} \n{X}_{S_1}.
	\end{equation}
\end{enumerate}
\end{definition} 

Notice that the MSCSP is a weakening of condition (i) in \cite[Lemma 8]{Oymak-Mohan-Fazel-Hassibi}, since we are only considering $X' = \Delta_1(\mathcal{A}X)$.
Also, the name of the MWP comes from its clear relationship to the definition of the Gelfand numbers/widths.
The following theorem is a matrix version of the Kashin-Temlyakov theorem \cite[Thm. 2.2]{Kashin-Temlyakov}:

\begin{theorem}
	For a linear operator $\mathcal{A} : M_N \to \R^m$, the MSCSP, MWCSP and MWP are equivalent (up to a change in the constants).
\end{theorem}

\begin{proof}
The MSCSP trivially implies the MWCSP, since $\rho_s(X)_{S_1} \le \n{X - 0}_{S_1} = \n{X}$.
Assume that $\mathcal{A}$ has the MWSCSP. Given $X \in \ker(\mathcal{A})$, note that $\mathcal{A}X = 0 = \mathcal{A}0$, so clearly $0 = \Delta_1(0) = \Delta_1(\mathcal{A}X)$ and thus from the MWSCSP we have $\n{X}_{S_2} \le Cs^{-1/2} \n{X}_{S_1}$, giving the MWP.
Assume now that we have the MWP, that is, that equation \eqref{eqn-MWP} holds.
If $s < \tfrac{1}{4}C^{-2}N/m$, from \cite[Thm. 2]{Oymak-Mohan-Fazel-Hassibi} (which is a matrix version of \cite[Thm. 2.1]{Kashin-Temlyakov}) we obtain
$$
\n{X - \Delta_1(\mathcal{A}X) }_{S_1} \le C' \rho_s(X)_{S_1}
$$
for $C' = 2(1 - 2\sqrt{C^2sm/N})^{-1}$. Since $X - \Delta_1(\mathcal{A}X) \in \ker\mathcal{A}$, the previous equation together with \eqref{eqn-MWP} imply
the MSCSP.
\end{proof}

The aforementioned Kashin-Temlyakov theorem says, in a nutshell, that the stability of sparse-vector recovery via $\ell_1$-minimization has limits imposed by the geometry of Banach spaces encoded in the appropriate Gelfand widths.
In the previous proposition, we showed a similar relationship relating the stability of low-rank recovery via nuclear norm minimization with some other Gelfand widths. 
As in the vector case, following \cite[Cor. 10.6]{Foucart-Rauhut}, there is a relationship between the geometry of $S_1^N$ and the stability of compressed sensing by any method. See Theorem \ref{theorem-compressive-widths} below for the precise statement.


\section{Stability of low-rank matrix recovery through Schatten $p$ quasi-norm minimization}\label{sec-FL}

In this technical section we prove a general result (a matrix version of the main theorem in \cite{Foucart-Lai}) that gives RIP-style conditions on the measuring map $\mathcal{A}$ that guarantee the stability of the Schatten $p$-norm minimization scheme.
For that we will need some notation:
Let $\alpha_s, \beta_s \ge 0$ be the best constants in the inequalities
$$
\alpha_s \n{Z}_{S_2} \le \n{\mathcal{A}Z}_{\ell_2} \le \beta_s \n{Z}_{S_2}, \quad \rank(Z) \le s.
$$
The results will be stated in terms of a quantity invariant under the change $\mathcal{A} \leftarrow c\mathcal{A}$, namely
$$
\gamma_{2s} := \frac{\beta_{2s}^2}{\alpha_{2s}^2} \ge 1.
$$
Note that this constant is related to the RIP constant, in fact
$$
\gamma_{2s} = \frac{1+\delta_{2s}}{1-\delta_{2s}}.
$$

Unlike in the rest of the paper, we will consider the more general situation of approximate recovery when measurements are moderately flawed, namely the problem
\begin{equation}\label{problem-Pptheta}\tag{$P_{p,\theta}$}
	\text{minimize} \n{Z}_{S_p} \quad \text{subject to }\quad \n{\mathcal{A}Z - y}_{\ell_2} \le \beta_{2s} \cdot \theta. 
\end{equation}
For simplicity, we will write ($P_p$) instead of ($P_{p,0}$).
Note that by a compactness argument, a solution of \eqref{problem-Pptheta} exists for any $0 < p \le 1$ and any $\theta \ge 0$.
The following theorem is a matrix version of \cite[Thm. 3.1]{Foucart-Lai}.
It gives conditions (in the spirit of the RIP) that guarantee not only the stability but also the robustness (that is, resistance to errors in the measurements) of the Schatten $p$-quasi-norm-minimization for low-rank matrix recovery.

\begin{theorem}\label{thm-Foucart-Lai}
	Given $0 < p \le 1$, if for some integer $t \ge s$
	\begin{equation}\label{eqn-thm-Foucart-Lai-hypothesis}
		\gamma_{2t}-1 < 4(\sqrt{2}-1)\left(\frac{t}{s}\right)^{1/p-1/2} 
	\end{equation}
	then a solution $X^*$ of \eqref{problem-Pptheta} approximates the original matrix $X$ with errors
	\begin{align}
		\n{X-X^*}_{S_p} &\le C_1 \rho_s(X)_{S_p} + D_1\cdot s^{1/p-1/2}\cdot \theta, \label{eqn-Foucart-Lai-estimate-1}\\
		\n{X-X^*}_{S_2} &\le C_2 \frac{\rho_s(X)_{S_p}}{t^{1/p-1/2}} + D_2\cdot \theta \label{eqn-Foucart-Lai-estimate-2},
	\end{align}
	where the constants $C_1$, $C_2$, $D_1$ and $D_2$ depend only on $p$, $\gamma_{2t}$ and the ratio $s/t$.
\end{theorem}

\begin{proof}
	We will need to recall some properties of the $S_p$-quasi-norm. Namely, for any matrices $U$ and $V$,
	\begin{equation}\label{eqn-Sp-quasinorm}
		\n{U}_{S_1} \le \n{U}_{S_p}, \quad \n{U}_{S_p} \le N^{1/p-1/2}\n{U}_{S_2}, \quad \n{U+V}^p_{S_p} \le \n{U}^p_{S_p} + \n{V}^p_{S_p}.
	\end{equation}
	
	\textbf{STEP 1:} Consequence of the assumption on $\gamma_{2t}$.\\
	We will consider certain matrix decompositions similar to the ones in \cite{Kong-Xiu}. Consider the singular value decomposition of $X$, given by
	$$
	X = U \diag(\lambda_i(X)) V^T
	$$
	where $U$, $V$ are unitary matrices and $\lambda(X) = (\lambda_1(X), \dotsc, \lambda_N(X))$ are the singular values of $X$ arranged in decreasing order.
	For any matrix $Z \in M_N$, we will consider a block decomposition of $Z$ with respect to $X$ as follows:
	let $U^TZV$ have the block form
	$$
	U^TZV = \begin{pmatrix} Z_{11} &Z_{12}\\
	Z_{21} &Z_{22} \end{pmatrix} 
	$$
	where $Z_{11}, Z_{12}, Z_{21}, Z_{22}$ are of sizes $s \times s$, $s \times (N-s)$, $(N-s) \times s$, $(N-s)\times (N-s)$, respectively.
	We then decompose $Z$ as $Z = Z_{(s)} + Z_{(s)}^c$ where
	\begin{align*}
		Z_{(s)} &= U \begin{pmatrix} Z_{11} &Z_{12}\\
		Z_{21} &0 \end{pmatrix} V^T \\
		Z_{(s)}^c &= U \begin{pmatrix} 0 &0\\
		0 &Z_{22} \end{pmatrix} V^T \\
	\end{align*}
	Furthermore, we now consider the singular value decomposition of $Z_{22}$ given by
	$$
	Z_{22} = P \diag(\lambda(Z_{22})) Q^T
	$$
	with $P$ and $Q$ being $(N-s)\times(N-s)$ unitary matrices, and $\lambda(Z_{22})$ is the vector of the $N-s$ singular values of $Z_{22}$ arranged in decreasing order.
	We decompose $\lambda(Z_{22})$ as a sum of vectors $Z_{T_i}$, each of sparsity at most $t$, where $T_1$ corresponds to the locations of the $t$ largest entries of $\lambda(Z_{22})$, $T_2$ to the locations of the next $t$ largest entries, and so on. For $i \ge 1$ we now define
	$$
	Z_{T_i} = U \begin{pmatrix} 0 &0\\
	0 &P \diag(\lambda_{T_i}(Z_{22}))Q^T \end{pmatrix} V^T,\\
	$$
	and denote $Z_{T_0} := Z_{(s)}$.
	
	We first observe that
	\begin{align}
		\n{Z_{T_0}}_{S_2}^2 + \n{Z_{T_1}}_{S_2}^2 &= \n{Z_{T_0} + Z_{T_1}}_{S_2}^2 \le \frac{1}{\alpha_{2t}^2}\n{ \mathcal{A}(Z_{T_0} + Z_{T_1}) }_{\ell_2}^2 \nonumber \\
		&= \frac{1}{\alpha_{2t}^2} \big\langle \mathcal{A}( Z - Z_{T_2} - Z_{T_3} - \cdots ) , \mathcal{A}(Z_{T_0} + Z_{T_1})  \big\rangle \nonumber \\
		&=  \frac{1}{\alpha_{2t}^2} \big\langle \mathcal{A}Z, \mathcal{A}(Z_{T_0} + Z_{T_1}) \big\rangle  +  \frac{1}{\alpha_{2t}^2} \sum_{k\ge 2} \big[ \big\langle \mathcal{A}(-Z_{T_k}) , \mathcal{A}Z_{T_0} \big\rangle + \big\langle \mathcal{A}(-Z_{T_k}) , \mathcal{A}Z_{T_1} \big\rangle \big] \label{eqn-step1-inner-products}
	\end{align}
	
	Let us renormalize the vectors $-Z_{T_k}$ and $Z_{T_0}$ so that their $S_2$-norms equal one by setting $Y_k := -Z_{T_k}/ \n{Z_{T_k}}_{S_2}$ and $Y_0 := Z_{T_0}/\n{Z_{T_0}}_{S_2}$.
	We then obtain, using the polarization identity
	\begin{align*}
		\frac{\big\langle \mathcal{A}(-Z_{T_k}) , \mathcal{A}Z_{T_0} \big\rangle}{\n{Z_{T_k}}_{S_2}\n{Z_{T_0}}_{S_2}} &= \pair{\mathcal{A}Y_k}{\mathcal{A}Y_0} 
		= \frac{1}{4}\big[ \n{\mathcal{A}(Y_k+Y_0)}_{\ell_2}^2 - \n{\mathcal{A}(Y_k-Y_0)}_{\ell_2}^2\big] \\
		&\le \frac{1}{4}\big[ \beta_{2t}^2\n{Y_k+Y_0}_{S_2}^2 - \alpha_{2t}^2\n{Y_k-Y_0}_{S_2}^2\big] = \frac{1}{2}[ \beta_{2t}^2 - \alpha_{2t}^2 ]. 
	\end{align*}
	An analogous argument with $T_1$ in place of $T_0$ allows us to conclude
	\begin{equation}\label{eqn-step1-polarization}
		\big\langle \mathcal{A}(-Z_{T_k}) , \mathcal{A}Z_{T_0} \big\rangle + \big\langle \mathcal{A}(-Z_{T_k}) , \mathcal{A}Z_{T_1} \big\rangle 
		\le \frac{\beta_{2t}^2 - \alpha_{2t}^2}{2}\n{Z_{T_k}}_{S_2}\big[ \n{Z_{T_0}}_{S_2} + \n{Z_{T_1}}_{S_2} \big].
	\end{equation}
	On the other hand, we have
	\begin{equation}\label{eqn-step1-cauchy-schwartz}
		\big\langle \mathcal{A}Z, \mathcal{A}(Z_{T_0} + Z_{T_1}) \big\rangle \le \n{\mathcal{A}Z}_{\ell_2}\cdot \n{\mathcal{A}(Z_{T_0} + Z_{T_1})}_{\ell_2}
		\le \n{\mathcal{A}Z}_{\ell_2}\cdot \beta_{2t}\big[\n{Z_{T_0}}_{S_2}+\n{Z_{T_1}}_{S_2}\big].
	\end{equation}
	Substituting the inequalities \eqref{eqn-step1-polarization} and \eqref{eqn-step1-cauchy-schwartz} into \eqref{eqn-step1-inner-products} we have
	$$
	\n{Z_{T_0}}_{S_2}^2 + \n{Z_{T_1}}_{S_2}^2 \le \left( \frac{\gamma_{2t}}{\beta_{2t}} \n{\mathcal{A}Z}_{\ell_2} + \frac{\gamma_{2t}-1}{2} \sum_{k \ge 2} \n{Z_{T_k}}_{S_2}  \right)\big[ \n{Z_{T_0}}_{S_2} + \n{Z_{T_1}}_{S_2} \big].
	$$
	If we set $c := \n{\mathcal{A}Z}_{\ell_2}\cdot\gamma_{2t}/\beta_{2t}$, $d := (\gamma_{2t}-1)/2$ and $\Sigma = \sum_{k \ge 2} \n{Z_{T_k}}_{S_2}$, the previous inequality is
	$$
	\n{Z_{T_0}}_{S_2}^2 - (c+d\Sigma)\n{Z_{T_0}}_{S_2} + \n{Z_{T_1}}_{S_2}^2 - (c+d\Sigma)\n{Z_{T_1}}_{S_2} \le 0,
	$$
	or equivalently,
	$$
	\left[ \n{Z_{T_0}}_{S_2} - \frac{c+d\Sigma}{2} \right]^2 + \left[ \n{Z_{T_1}}_{S_2} - \frac{c+d\Sigma}{2} \right]^2 \le \frac{(c+d\Sigma)^2}{2}.
	$$
	by getting rid of the second squared term, this easily implies
	\begin{equation}\label{eqn-step1-easy-bound-from-squares}
		\n{Z_{T_0}}_{S_2} \le \frac{c+d\Sigma}{2} + \frac{c+d\Sigma}{\sqrt{2}} = \frac{1+\sqrt{2}}{2}(c+d\Sigma).
	\end{equation}
	By H\"older's inequality (see \eqref{eqn-Sp-quasinorm}) we get
	\begin{equation}\label{eqn-step1-Holder}
		\n{Z_{T_0}}_{S_p} \le s^{1/p-1/2}\n{Z_{T_0}}_{S_2} \le  s^{1/p-1/2}\frac{1+\sqrt{2}}{2}(c+d\Sigma).
	\end{equation}
	We now proceed to bound $\Sigma$.
	For $k \ge 2$, let $\eta$, $\eta'$ be singular values of $Z_{T_k}$, $Z_{T_{k-1}}$, respectively. By definition, we must have $\eta \le \eta'$. Raising to the $p$-th power and averaging over all singular values of $Z_{T_{k-1}}$, $\eta^p \le t^{-1} \n{Z_{T_{k-1}}}^p_{S_p}$, and hence $\eta^2 \le t^{-2/p} \n{Z_{T_{k-1}}}_{S_p}^{2}$.
	Adding over all singular values of $Z_{T_k}$ and taking the square root, this yields $\n{Z_{T_k}}_{S_2} \le t^{1/2-1/p} \n{Z_{T_{k-1}}}_{S_p}$. Therefore,
	$$
	\Sigma = \sum_{k \ge 2} \n{Z_{T_k}}_{S_2} \le t^{1/2-1/p} \sum_{k \ge 1} \n{Z_{T_k}}_{S_p} \le t^{1/2-1/p} \Bigg[ \sum_{k \ge 1} \n{Z_{T_k}}_{S_p}^p \Bigg]^{1/p} = t^{1/2-1/p} \n{Z_{(s)}^c}_{S_p}.
	$$
	 Combining the above inequality with \eqref{eqn-step1-Holder}, we obtain
	\begin{equation}\label{eqn-step1-partial-conclusion}
		\n{Z_{(s)}}_{S_p} \le \frac{\lambda}{2\beta_{2t}} \cdot\n{\mathcal{A}Z}_{\ell_2}\cdot s^{1/p-1/2} + \mu \cdot \n{Z_{(s)}^c}_{S_p}.
	\end{equation}
	where the constants $\lambda$ and $\mu$ are given by
	$$
	\lambda := (1+\sqrt{2})\gamma_{2t} \text{ and } \mu := \frac{1}{4}(1+\sqrt{2})(\gamma_{2t}-1) \left(\frac{s}{t}\right)^{1/p-1/2}.
	$$
	Note that the assumption on $\gamma_{2t}$ translates into the inequality $\mu < 1$.
	\\
	\\
	\textbf{STEP 2:}
	From now on let $Z := X - X^*$. \\
	Because $X^*$ is a minimizer of \eqref{problem-Pptheta}, we have
	\begin{equation}\label{eqn-step2-minimizer}
	\n{X^*}_{S_p}^p \le \n{X}_{S_p}^p.	
	\end{equation}	
	From \cite[Lemma 2.2]{Kong-Xiu}, whenever $B,C \in M_N$ satisfy $B^TC = 0$ and $BC^T= 0$ one has $\n{B+C}_{S_p}^p = \n{B}_{S_p}^p+\n{C}_{S_p}^p$.
	In particular, note that
	\begin{equation}\label{eqn-step2-orthogonality-of-the-norms}
		\n{X}_{S_p}^p = \n{X_{(s)}}_{S_p}^p + \n{X_{(s)}^c}_{S_p}^p \text{ and } \n{X_{(s)} - Z_{(s)}^c}_{S_p}^p = \n{X_{(s)}}_{S_p}^p + \n{Z_{(s)}^c}_{S_p}^p.
	\end{equation}   
	From the $p$-triangle inequality (see \eqref{eqn-Sp-quasinorm}), since
	$$
	X_{(s)} - Z_{(s)}^c = X - Z - X_{(s)}^c + Z_{(s)} = X^* - X_{(s)}^c + Z_{(s)},
	$$
	we get
	$$
	\n{ X_{(s)} - Z_{(s)}^c  }_{S_p}^p \le \n{X^*}_{S_p}^p + \n{X_{(s)}^c}_{S_p}^p + \n{Z_{(s)}}_{S_p}^p.
	$$
	Together with \eqref{eqn-step2-minimizer} and both equalities in \eqref{eqn-step2-orthogonality-of-the-norms}, this yields
	$$
	\n{X_{(s)}}_{S_p}^p + \n{Z_{(s)}^c}_{S_p}^p \le  \n{X_{(s)}}_{S_p}^p + \n{X_{(s)}^c}_{S_p}^p \n{X_{(s)}^c}_{S_p}^p + \n{Z_{(s)}}_{S_p}^p.
	$$
	After a cancellation and noticing that $\n{X_{(s)}^c}_{S_p}^p = \rho_s(X)_{S_p}^p$, we obtain
	\begin{equation}\label{eqn-part2-conclusion}
		\n{Z_{(s)}^c}_{S_p}^p \le 2 \rho_s(X)_{S_p}^p + \n{Z_{(s)}}_{S_p}^p.
	\end{equation}
	
	\textbf{STEP 3:} Error estimates. \\
	We first note the bound
	$$
	\n{\mathcal{A}Z}_{\ell_2} = \n{\mathcal{A}X - \mathcal{A}X^*}_{\ell_2} \le \n{\mathcal{A}X - y}_{\ell_2} + \n{y - \mathcal{A}X^*}_{\ell_2} \le 2\beta_{2s}\cdot \theta.
	$$
	For the $S_p$-error, we combine the estimates in \eqref{eqn-step1-partial-conclusion} and \eqref{eqn-part2-conclusion} to obtain
	$$
	\n{Z_{(s)}^c}_{S_p}^p \le  2 \rho_s(X)_{S_p}^p + \lambda^p \cdot s^{1-p/2}\cdot\theta^p + \mu^p \cdot \n{Z_{(s)}^c}_{S_p}^p.
	$$
	As a consequence of $\mu<1$, we have
	$$
	\n{Z_{(s)}^c}_{S_p}^p \le  \frac{2}{1-\mu^p} \rho_s(X)_{S_p}^p + \frac{\lambda^p}{1-\mu^p} \cdot s^{1-p/2}\cdot\theta^p.
	$$
	Using the estimate \eqref{eqn-step1-partial-conclusion} once again, we can derive that
	\begin{align*}
		\n{Z}_{S_p}^p &\le \n{Z_{(s)}}_{S_p}^p + \n{Z_{(s)}^c}_{S_p}^p \le (1+\mu^q)\cdot\n{Z_{(s)}^c}_{S_p}^p + \lambda^p \cdot s^{1-p/2}\cdot\theta^p \\
		&\le \frac{2}{1-\mu^p}(1+\mu^p) \rho_s(X)_{S_p}^p + \frac{2\lambda^p}{1-\mu^p}\cdot s^{1-p/2}\cdot\theta^p \\
		&\le 2^{1-1/p}\left[ \frac{2^{1/p}}{(1-\mu^p)^{1/p}}(1+\mu^p)^{1/p} \rho_s(X)_{S_p} + \frac{2^{1/p}\lambda}{(1-\mu^p)^{1/p}}\cdot s^{1/p-1/2}\cdot\theta \right]^p
	\end{align*}
	where we have used the inequality $(a^p + b^p)^{1/p} \le 2^{1/p-1}(a+b)$ for $a,b \ge 0$.
	The desired estimate \eqref{eqn-Foucart-Lai-estimate-1} follows with 
	$$
	C_1 := \frac{2^{2/p-1}(1+\mu^p)^{1/p}}{(1-\mu^p)^{1/p}}, D_1 := \frac{2^{2/p-1}\lambda}{(1-\mu^p)^{1/p}}.
	$$
	For the $S_2$-error, let us observe that the bound in
	 \eqref{eqn-step1-easy-bound-from-squares} also holds if we replace $\n{Z_{T_0}}_{S_2}$ by $\n{Z_{T_1}}_{S_2}$, and hence
	$$
	\n{Z}_{S_2} = \left[ \sum_{k \ge 0} \n{Z_{T_k}}_{S_2}^2 \right]^{1/2}
	\le \sum_{k \ge 0} \n{Z_{T_k}}_{S_2} \le (1+\sqrt{2}) \cdot (c+d\Sigma) + \Sigma
	\le \nu \cdot \Sigma + 2\lambda \cdot \theta,
	$$
	where $\nu = (\lambda + 1-\sqrt{2})/2$.
	We also have that
\begin{multline*}
	\Sigma \le t^{1/2-1/p} \n{Z_{(s)}^c}_{S_p} 
	\le t^{1/2-1/p} \left[ \frac{2}{1-\mu^p} \rho_s(X)_{S_p}^p + \frac{\lambda^p}{1-\mu^p} \cdot s^{1-p/2}\cdot\theta^p \right]^{1/p}\\
	\le t^{1/2-1/p}2^{1/p-1} \left[ \frac{2^{1/p}}{(1-\mu^p)^{1/p}} \rho_s(X)_{S_p} + \frac{\lambda}{(1-\mu^p)^{1/p}} \cdot s^{1/p-1/2}\cdot\theta \right],
\end{multline*}
	and hence we conclude that
	$$
	\n{Z}_{S_2} \le \nu t^{1/2-1/p} 2^{1/p-1} \left[ \frac{2^{1/p}}{(1-\mu^p)^{1/p}} \rho_s(X)_{S_p} + \frac{\lambda}{(1-\mu^p)^{1/p}} \cdot s^{1/p-1/2}\cdot\theta \right]
	+ 2\lambda \theta.		
	$$
This gives the estimate \eqref{eqn-Foucart-Lai-estimate-2} with
$$
C_2 = \frac{2^{2/p-2}(\lambda+1-\sqrt{2})}{(1-\mu^p)^{1/p}}, \qquad D_2 = \frac{2^{1/p-2}\lambda(\lambda+1-\sqrt{2})}{(1-\mu^p)^{1/p}} + 2\lambda.
$$
\end{proof}

As consequences of Theorem \ref{thm-Foucart-Lai}, we obtain two corollaries that are matrix versions of the ones in \cite{Foucart-Lai}.
The first one corresponds to the case of exact recovery.

\begin{corollary}
	Given $0 < p \le 1$, if
	$$
	\gamma_{2t} - 1 < 4 (\sqrt{2} - 1) \left( \frac{t}{s} \right)^{1/p-1/2} \quad\text{ for some integer } t \ge s,
	$$
	then every rank $s$ matrix is exactly and stably recovered by solving ($P_p$).
\end{corollary}

The second one deals with the special case of nuclear norm minimization.

\begin{corollary}
	Under the assumption that $\gamma_{2s} < 4\sqrt{2}-3 \approx 2.6569,$
	every rank $s$ matrix is exactly and stably recovered by solving ($P_1$).
\end{corollary}

This last Corollary is clearly related to existing results on the RIP, it corresponds to the condition $\delta_{2s} < 2(3-\sqrt{2})/7 \approx 0.4531$. 
Note that for the specific case of $p=1$ this condition is not the best possible:
a very recent result of Cai and Zhang \cite{Cai-Zhang} shows that the optimal condition to have exact recovery of rank $s$ matrices via nuclear norm minimization is in fact $\delta_{2s} < 1/\sqrt{2} \approx 0.7071$.
Another recent result similar to our Theorem \ref{thm-Foucart-Lai} is
\cite[Thm. 6]{Liu-Huang-Chen} (which in turn generalizes results of Lee and Bresler \cite{Lee-Bresler-admira,Lee-Bresler-nuclear}), where they get a conclusion of the same form as \eqref{eqn-Foucart-Lai-estimate-2} but requiring a stronger hypothesis.
Finally, note also that Theorem \ref{thm-existence-RIP-matrices} guarantees the existence of maps $\mathcal{A}$ satisfying the hypothesis of Theorem \ref{thm-Foucart-Lai}.

\section{The Gelfand widths of $S_p$-balls for $0 < p \le 1$}

In this section we calculate the Gelfand numbers
$$
c_m(id : S_p^N \to S_q^N) 
$$
for $0 < p \le 1$, $p < q \le 2$.
This can be considered as a noncommutative version of the results from \cite{Foucart-Pajor-Rauhut-Ullrich},
where they use compressed sensing ideas to calculate the corresponding Gelfand numbers
$$
c_m(id : \ell_p^N \to \ell_q^N).
$$
Inspired by their approach, our proof is based on low-rank matrix recovery ideas.

Our main result is the following (compare to \cite[Thm. 1.1]{Foucart-Pajor-Rauhut-Ullrich}).

\begin{theorem}\label{thm-Gelfand-widths-of-Schatten-classes}
	For $0 < p \le 1$ and $p < q \le 2$, if $1 \le m < N^2$, then
	$$
	d^m(B^N_{S_p}, S_q^N) \asymp_{p,q} \min\left\{1, \frac{N}{m} \right\}^{1/p-1/q}.
		$$
	and, if $p < 1$, 
	$$
	d^m(B^N_{S_{p,\infty}}, S_q^N) \asymp_{p,q} \min\left\{1, \frac{N}{m} \right\}^{1/p-1/q}.
	$$
\end{theorem}

Before the proof, let us go through some preliminaries.
Recall that it is classical to show that, for $q>p$,
\begin{align}
	\rho_s(X)_{S_q} &\le \frac{1}{s^{1/p-1/q}}\n{X}_{S_p}, \label{eqn-trivial-best-appxn-Sp} \\
	\rho_s(X)_{S_q} &\le \frac{D_{p,q}}{s^{1/p-1/q}}\n{X}_{S_{p,\infty}}, \qquad D_{p,q} := (q/p-1)^{-1/q}. \label{eqn-trivial-best-appxn-weak-Sp}
\end{align}

\subsection{Lower bounds}

In this section we prove a result that will easily imply the desired lower bounds in Theorem \ref{thm-Gelfand-widths-of-Schatten-classes}.
It is a matrix version of \cite[Thm. 2.1]{Foucart-Pajor-Rauhut-Ullrich} and, just like in their result, we note that the restriction $q \le 2$ is not imposed here.

\begin{proposition}\label{prop-lower-bounds}
	For $0 < p \le 1$ and $p < q \le \infty$, there exists a constant $c_{p,q}>0$ such that
	$$
	d^m(B^N_{S_p}, S_q^N) \ge c_{p,q} \min\left\{1, \frac{N}{m}\right\}^{1/p-1/q}
	$$
\end{proposition}

\begin{proof}
	With $c = (1/2)^{2/p-1/q}$ and $\mu = \min\{1,N/2m\}$, we are going to prove that
	$$
	d^m(B^N_{S_p}, S_q^N) \ge c \mu^{1/p-1/q}.
	$$
	We proceed by contradiction, assuming that $d^m(B^N_{S_p}, S_q^N) < c \mu^{1/p-1/q}$.
	This implies the existence of a linear map $\mathcal{A} : M_N \to \R_m$ such that for all $V \in \ker(\mathcal{A}) \setminus \{0\}$,
	$$
	\n{V}_{S_q} < c \mu^{1/p-1/q} \n{V}_{S_p}.
	$$
	For a fixed $V \in \ker(\mathcal{A}) \setminus \{0\}$, in view of the inequalities $\n{V}_{S_p} \le N^{1/p-1/q}\n{V}_{S_q}$ and $c \le (1/2)^{1/p-1/q}$, we derive $1 < (\mu N/2)^{1/p-1/q}$, so $1 \le 1/\mu < N/2$.
	We then define $s := \lfloor 1/\mu \rfloor \ge 1$, so $2s < N$ and
	$$
	\frac{1}{2\mu} < s \le \frac{1}{\mu}.
	$$
	Now for $V \in \ker(\mathcal{A}) \setminus \{0\}$, 
	$$
	\n{V_{[2s]}}_{S_p} \le (2s)^{1/p-1/q}\n{V_{[2s]}}_{S_q} \le (2s)^{1/p-1/q}\n{V}_{S_q} < c(2s\mu)^{1/p-1/q}\n{V}_{S_p} \le \frac{1}{2^{1/p}}\n{V}_{S_p}
	$$
	and therefore, using that $\n{V}_{S_p}^p = \n{V_{[2s]}}^p_{S_p} + \n{V-V_{[2s]}}^p_{S_p}$, we conclude
	$$
	\n{V_{[2s]}}^p_{S_p} \le \n{V-V_{[2s]}}^p_{S_p}.
	$$
	This means that $\mathcal{A}$ satisfies the sufficient conditions in \cite[Thm. 3]{Oymak-Mohan-Fazel-Hassibi}, which implies that Schatten $p$-quasinorm minimization gives exact recovery of rank $s$ matrices.
	By well-known arguments (see, for example, the discussion after the statement of theorem 2.3 in \cite{Candes-Plan}), this gives
	$$
	m \ge Ns > N\frac{1}{2\mu} \ge \frac{N}{2}\frac{1}{\mu} \ge  \frac{N}{2}\frac{2m}{N} = m,
	$$
	a blatant contradiction.
	
\end{proof}

\subsection{Upper bounds}

In this subsection we establish a result from which the desired upper bounds in Theorem \ref{thm-Gelfand-widths-of-Schatten-classes} will follow easily.
The proof relies on low-rank matrix recovery methods, and the reader will notice similarities with the proof of Theorem \ref{thm-Foucart-Lai}.
It should be mentioned that the bound for the case $p \ge 1$ follows easily from the result of Carl and Defant \cite{Carl-Defant-tensor-products} mentioned in the introduction together with an interpolation argument, but the bound for the case $p<1$ is new (as far as the authors know).
Our result is a matrix version of \cite[Thm. 3.2]{Foucart-Pajor-Rauhut-Ullrich},
but the essence of the argument can be traced back to Donoho \cite[Thm. 9]{Donoho}.
As in the case of the result of \cite{Foucart-Pajor-Rauhut-Ullrich}, it is interesting to note that even when $p < 1$, an optimal reconstruction map $\Delta$ 
for the realization of the number $E_m(B_{S_{p,\infty}}^N,S_q^N)$ can be chosen to be the $S_1$-minimization mapping, at least when $q \geq  1$.

\begin{theorem}\label{thm-upper-bounds}
	For $0<p < 1$ and $p<q\leq 2$, there exists a linear map $\mathcal{A} : M_N \to \R^m$ such that, with $r = \min\{1,q\}$,
	\[
	\sup_{X \in B_{p,\infty}^N} \n{X - \Delta_r(\mathcal{A}X)}_{S_q} \leq C_{p,q} \min \Big\{1, \frac{N}{m} \Big\}^{1/p-1/q},
	\]
	where $C_{p,q}>0$ is a constant that depends only on $p$ and $q$.
\end{theorem}

\begin{proof} 
Let $C$ be the constant in \eqref{mRIP} relative to the RIP 
associated with $\delta=1/3$, say.

{\em Case 1:} $m \ge C N$.\\
We define $s \ge 1$ as the largest integer smaller than $\frac{m}{CN}$, so that
\begin{equation}
\label{DefS}
\frac{m}{2CN}< s \le \frac{m}{CN}.
\end{equation}
Let $t=2s$.
It is then possible to find a linear map $\mathcal{A} : M_N \to \R^{m}$ with $\delta_{t}(A) \le \delta$.
In particular, we have $\delta_{s}(A) \le \delta$.
Now, given $Z := X - \Delta_r(\mathcal{A}X) \in \ker \mathcal{A}$, we decompose $Z$ into matrices $Z_{T_1}, Z_{T_2}, Z_{T_3}, \dotsc$ of rank at most $s$
by taking the $s$ largest singular values of $Z$ for $Z_{T_1}$, then the next $s$ largest ones for $Z_{T_2}$ and so on. 

This easily implies $\big( \n{Z_{T_k}}_{S_2}^2 / s \big)^{1/2} \le \big(  \n{Z_{T_{k-1}}}_{S_r}^r/s \big)^{1/r}$, i.e.,
\begin{equation}
\label{CompSkSk-1}
\n{Z_{T_k}}_{S_2} \le \frac{1}{s^{1/r-1/2}} \n{Z_{T_{k-1}}}_{S_r},
\qquad k \ge 2.
\end{equation}
Using the $r$-triangle inequality, we have
$$
\n{Z}_{S_q}^r =  \n{ \sum_{k \ge 1} Z_{T_k} }_{S_q}^r
\le \sum_{k \ge 1}  \n{Z_{T_k} }_{S_q}^r
\le \sum_{k \ge 1} \big(s^{1/q-1/2} \n{Z_{T_k}}_{S_2} \big)^{r}
\le \sum_{k \ge 1} \Big(\frac{s^{1/q-1/2}}{\sqrt{1-\delta}} \n{\mathcal{A}Z_{T_k}}_{\ell_2}  \Big)^r. 
$$
The fact that $Z \in  \ker \mathcal{A}$ implies $\mathcal{A}Z_{T_1} =  -\sum_{k \ge 2} \mathcal{A} Z_{T_k} $.
It follows that
\begin{eqnarray*}
\n{Z}_{S_q}^r  & \le &
\Big(\frac{s^{1/q-1/2}}{\sqrt{1-\delta}} \Big)^r \Big( \sum_{k \ge 2}  \n{\mathcal{A}Z_{T_k}}_{\ell_2} \Big)^r
+ \Big(\frac{s^{1/q-1/2}}{\sqrt{1-\delta}} \Big)^r \sum_{k \ge 2}  \n{\mathcal{A}Z_{T_k}}_{\ell_2}^r\\
& \le & 2 \Big(\frac{s^{1/q-1/2}}{\sqrt{1-\delta}} \Big)^r \sum_{k \ge 2}  \n{\mathcal{A}Z_{T_k}}_{\ell_2}^r
\le 2 \Big(\sqrt{\frac{1+\delta}{1-\delta}} s^{1/q-1/2}\Big)^r  \sum_{k \ge 2} \n{Z_{T_k}}_{S_2}^r.
\end{eqnarray*}
We then derive, using the inequality \eqref{CompSkSk-1}, 
$$
\n{Z}_{S_q}^r  \le 
2 \Big(\sqrt{\frac{1+\delta}{1-\delta}} \frac{1}{s^{1/r-1/q}}\Big)^{r} \sum_{k \ge 1} \n{Z_{T_k}}_{S_r}^r .
$$
In view of the choice $\delta=1/3$ and of \eqref{DefS}, we deduce
\begin{equation}
\label{Eq1LastPf}
\n{ X-\Delta_r(\mathcal{A}X) }_{S_q} \le 2^{1/r} \sqrt{2} \Big( \frac{2CN}{m} \Big)^{1/r-1/q} \n{ X-\Delta_r(\mathcal{A}X) }_{S_r}\,.
\end{equation}
Moreover, in view of $\delta_{2s} \leq 1/3$ and of Theorems \ref{thm-Foucart-Lai} and \ref{thm-existence-RIP-matrices},
there exists a constant $C_1>0$ such that 
\begin{equation}
\label{Eq2LastPf}
\n{X - \Delta_r(\mathcal{A}X)}_{S_r} \le (C_1)^{1/r} \rho_s(x)_{S_r}.
\end{equation}
Finally, using \eqref{eqn-trivial-best-appxn-weak-Sp} and \eqref{DefS}, we have 
\begin{equation}
\label{Eq3LastPf}
\rho_s(X)_{S_r} \le  \, \frac{D_{p,r}}{s^{1/p-1/r}} 
\le \, D_{p,r}\Big( \frac{2CN}{m} \Big)^{1/p-1/r}\,.
\end{equation}
Putting \eqref{Eq1LastPf}, \eqref{Eq2LastPf}, and \eqref{Eq3LastPf} together,
we obtain, for any $x \in B_{p,\infty}^N$, 
$$
\n{ X-\Delta_r(\mathcal{A}X) }_{S_q} \le 2^{1/r} \sqrt{2} C_1^{1/r} D_{p,r} \Big( \frac{2CN}{m} \Big)^{1/p-1/q}. 
$$

{\em Case 2:} $m \le CN$.\\
We simply choose the map $\mathcal{A}$ as the zero map.
Then, for any $X \in B_{S_{p,\infty}}^N$, we have
$$
\n{X - \Delta_r(\mathcal{A}X)}_{S_q} = \n{X}_{S_q} \le D_{p,q} \n{X}_{p,\infty} \le D_{p,q},
$$
for some constant $D_{p,q}>0$.

This completes the proof.
\end{proof}

\begin{remark}\label{remark-upper-bound-p-equals-one}
	When $p=1$, the same proof but using inequality \eqref{eqn-trivial-best-appxn-Sp} instead of \eqref{eqn-trivial-best-appxn-weak-Sp}
	gives the following: for $1<q\leq 2$, there exists a linear map $\mathcal{A} : M_N \to \R^m$ such that,
	\[
	\sup_{X \in B_{1}^N} \n{X - \Delta_1(\mathcal{A}X)}_{S_q} \leq C_{q} \min \Big\{1, \frac{N}{m} \Big\}^{1-1/q},
	\]
	where $C_{q}>0$ is a constant that depends only on $q$.
\end{remark}

\subsection{Proof of theorem \ref{thm-Gelfand-widths-of-Schatten-classes}}

\begin{proof}
First, an observation.
As in the vector case, the simple inclusion $B^N_{S_p} \subseteq B^N_{S_{p,\infty}}$ implies
$$
d^m(B^N_{S_p}, S_q^N) \le d^m(B^N_{S_{p,\infty}}, S_q^N),
$$
hence it suffices to show lower bounds for $d^m(B^N_{S_p}, S_q^N)$ and upper bounds for $d^m(B^N_{S_{p,\infty}}, S_q^N)$.

The lower bounds follow immediately from Proposition \ref{prop-lower-bounds}.
When $0<p<1$, the upper bounds follow from Theorem \ref{thm-upper-bounds}.
For $p=1$, the upper bound when $1 \le m \le N$ follows from the trivial inequality $\n{X}_{S_q} \le \n{X}_{S_1}$,
whereas when $N \le m \le N^2$ it follows from Remark \ref{remark-upper-bound-p-equals-one}.
\end{proof}

\subsection{Relation to compressive widths}

As promised after the proof of our matrix version of the Kashin-Temlyakov theorem, the relationship between the Banach space geometry of the finite-dimensional Schatten $p$-classes and matrix recovery goes beyond the norm minimization scheme.
Below we use the notation from \cite[Sec. 10.1]{Foucart-Rauhut}:
the quantities $E^m$ and $E^m_{\ada}$ measure the worst-case reconstruction errors of optimal measurement/reconstruction schemes in the nonadaptive and adaptive settings, respectively.

\begin{theorem}\label{theorem-compressive-widths}
	For $0 < p \le 1$ and $p < q \le 2$, if $1 \le m < N^2$ then the adaptive and nonadaptive compressive widths satisfy
	$$
	E^m_{\ada}( B_{S_p^N}, S_q^N ) \asymp_{p,q} E^m( B_{S_p^N}, S_q^N ) \asymp_{p,q} \min\left\{1, \frac{N}{m} \right\}^{1/p-1/q}.
	$$
\end{theorem}

\begin{proof}
	Since $-B_{S_p^N} = B_{S_p^N}$ and $B_{S_p^N}+B_{S_p^N} \subseteq 2^{1/p}B_{S_p^N}$, \cite[Thm. 10.4]{Foucart-Rauhut}
	implies
	$$
	d^m( B_{S_p^N}, S_q^N ) \le E^m_{\ada}( B_{S_p^N}, S_q^N ) \le E^m( B_{S_p^N}, S_q^N ) \le 2^{1/p}d^m( B_{S_p^N}, S_q^N )
	$$
	But now, since $d^m(B_{S_p^N}, S_q^N) = c_{m+1}(id : S_p^N \to S_q^N)$, an appeal to Theorem \ref{thm-Gelfand-widths-of-Schatten-classes}
	finishes the proof. 
\end{proof}

In the $\ell_p$ case the lower estimate is of particular importance in compressed sensing, since it allows one to prove lower bounds for the number of measurements required to stably recover $s$-sparse vectors in $\R^N$.
In the matrix case, that is no longer the case.
Trying it only gives that (under certain conditions), the minimum number of measurements $m$ required to stably recover rank $s$ matrices in $M_N$ is $\ge C Ns$, which is not an improvement over the information-theoretical limit.
The reason behind this is that, unlike in the $\ell_p$ case, there are compressed sensing algorithms (including norm minimization) that give stability with a number of measurements of that order \cite{Candes-Plan}.

\section*{Acknowledgements}

We would like to thank Rachel Ward for suggesting the reference \cite{Foucart-Lai},
and also thank the Workshop in Analysis in Probability at Texas A\&{}M University.
The first author was partially supported by NSF grant DMS-1400588.

\bibliographystyle{plain}

\end{document}